\documentclass[12pt]{amsart}
\usepackage{amssymb}
\usepackage[shortlabels]{enumitem}
\usepackage[all]{xy}
\usepackage{xcolor}
\usepackage{mathrsfs}
\usepackage[margin=1in]{geometry} 
\usepackage[bookmarks, bookmarksdepth=2, colorlinks=true, linkcolor=blue, citecolor=blue, urlcolor=blue]{hyperref}
\usepackage{eucal}
\usepackage{contour}
\usepackage[normalem]{ulem}
\usepackage{dsfont}

\makeatletter
\@addtoreset{equation}{section}
\makeatother

\numberwithin{equation}{section}
\newtheorem{theorem}[equation]{Theorem}

\newtheorem{proposition}[equation]{Proposition}
\newtheorem{lemma}[equation]{Lemma}
\newtheorem{corollary}[equation]{Corollary}

\theoremstyle{definition}
\newtheorem{rmk}[equation]{Remark}
\newenvironment{remark}[1][]{\begin{rmk}[#1] \pushQED{\qed}}{\popQED \end{rmk}}
\newtheorem{eg}[equation]{Example}
\newenvironment{example}[1][]{\begin{eg}[#1] \pushQED{\qed}}{\popQED \end{eg}}
\newtheorem{defnaux}[equation]{Definition}
\newenvironment{definition}[1][]{\begin{defnaux}[#1]\pushQED{\qed}}{\popQED \end{defnaux}}

\newcommand{\bC}{\mathbf{C}}

\newcommand{\fC}{\mathfrak{C}}

\newcommand{\sE}{\mathscr{E}}
\newcommand{\bF}{\mathbf{F}}

\newcommand{\fF}{\mathfrak{F}}

\newcommand{\sF}{\mathscr{F}}

\newcommand{\bN}{\mathbf{N}}

\newcommand{\bQ}{\mathbf{Q}}

\newcommand{\bS}{\mathbf{S}}

\newcommand{\fS}{\mathfrak{S}}

\newcommand{\bZ}{\mathbf{Z}}

\newcommand{\arxiv}[1]{\href{http://arxiv.org/abs/#1}{{\tiny\tt arXiv:#1}}}
\newcommand{\DOI}[1]{\href{http://doi.org/#1}{\color{purple}{\tiny\tt DOI:#1}}}

\contourlength{1pt}

\contourlength{0.8pt}
\newcommand{\myuline}[1]{%
  \uline{\phantom{#1}}%
  \llap{\contour{white}{#1}}%
}
\DeclareMathOperator{\uRep}{\text{\myuline{\rm Rep}}}

\renewcommand{\phi}{\varphi}
\DeclareMathOperator{\Aut}{Aut}

\DeclareMathOperator{\acl}{acl}

\newcommand{\defn}[1]{\emph{#1}}
\newcommand{\bone}{\mathbf{1}}
\DeclareMathOperator{\type}{\mathrm{tp}}

\DeclareMathOperator{\GEN}{\mathtt{Gen}}
\DeclareMathOperator{\FMM}{\mathtt{FMM}}
\DeclareMathOperator{\AUT}{\mathtt{Aut}}
\DeclareMathOperator{\IND}{\mathtt{Ind}}
\DeclareMathOperator{\DIS}{\mathtt{DIS}}

\title{Upper bounds for measures on distal classes}

\author{Ilia Nekrasov}
\address{Department of Mathematics, University of Michigan, Ann Arbor, MI, USA}
\email{\href{mailto:inekras@umich.edu}{inekras@umich.edu}}

\author{Andrew Snowden}
\address{Department of Mathematics, University of Michigan, Ann Arbor, MI, USA}
\email{\href{mailto:asnowden@umich.edu}{asnowden@umich.edu}}
\urladdr{\url{http://www-personal.umich.edu/~asnowden/}}
\thanks{AS was supported by NSF grant DMS-2301871.}

\date{July 25, 2024}

\begin{document}

\begin{abstract}
In recent work, Harman and Snowden introduced a notion of measure on a Fra\"iss\'e class $\fF$, and showed how such measures lead to interesting tensor categories. Constructing and classifying measures is a difficult problem, and so far only a handful of cases have been worked out. In this paper, we obtain some of the first general results on measures. Our main theorem states that if $\fF$ is distal (in the sense of Simon), and there are some bounds on automorphism groups, then $\fF$ admits only finitely many measures; moreover, we give an effective upper bound on their number. For example, if $\fF$ is the class of ``$s$-dimensional permutations'' (finite sets equipped with $s$ total orders), we show that the number of measures is bounded above by approximately $\exp(\exp(s^2 \log{s}))$.
\end{abstract}

\maketitle
\tableofcontents

\section{Introduction}

\subsection{Overview}

Let $\fF$ be a Fra\"iss\'e class (see Definition~\ref{defn:fraisse}). In \cite{repst}, we introduced a notion of \defn{measure} on $\fF$ (see Definition~\ref{defn:Fmeas}), and used measures to construct tensor categories. We recall two particularly interesting examples:
\begin{enumerate}
\item Suppose $\fF$ is the class of finite (unstructured) sets. There is then a 1-parameter family $\mu_t$ of complex measure on $\fF$. The tensor category constructed from $\mu_t$ is essentially Deligne's interpolation category $\uRep(\fS_t)$ of the symmetric group \cite{Deligne}.
\item Suppose $\fF$ is the class of finite totally ordered sets. Then $\fF$ admits exactly four measures. One of these measures leads to the remarkable Delannoy category, studied in depth in \cite{line}. See \cite{circle} for a closely related case.
\end{enumerate}
Determing the measures on $\fF$ is an important problem for applications to tensor categories. So far, only a handful of assorted cases have been worked out; see \cite{arboreal, repst, Kriz, thesis, colored, cantor, homoperm}.

In this paper, we obtain some of the first general results on measures. Our main theorem states that if $\fF$ is a \defn{distal class} (see Theorem~\ref{mainthm}(a) and \S \ref{ss:marked}), and there are some bounds on automorphism groups, then $\fF$ has finitely many measures; moreover, we give an effective upper bound on their number. This explains the qualitative differences between the above two examples: (a) is not distal, while (b) is. We discuss a few more examples in \S \ref{ss:ex}.

\subsection{Results}

We now precisely state our results. To this end, we let $\Theta(\fF)$ denote the ring carrying the universal measure for $\fF$ (see \S \ref{ss:Fmeas}). Giving a measure for $\fF$ valued in a ring $k$ is equivalent to giving a ring homomorphism $\Theta(\fF) \to k$, and so describing all measures for $\fF$ is equivalent to computing the single ring $\Theta(\fF)$. Our main theorem provides much information about this ring:

\begin{theorem} \label{mainthm}
Suppose $\fF$ satisfies the following two conditions\footnote{In all examples we know, the first condition implies the second.}:
\begin{enumerate}
\item Distality: there are finitely many minimal marked structures (\S \ref{ss:marked}).
\item Bounded automorphism groups: there is a positive integer $m$ such that for any $X \in \fF$, the order of $\Aut(X)$ divides some power of $m$.
\end{enumerate}
We then have a ring isomorphism
\begin{displaymath}
\Theta(\fF)[1/m] = \bZ[1/m]^n
\end{displaymath}
for some $n$. Moreover, if $r$ is the number of minimal marked structures and $p$ is the smallest prime not dividing $m$ then $n \le p^r$.
\end{theorem}

The quantity $n$ is the number of complex measures $\fF$ admits. Thus the theorem shows that, under the assumed hypothesis, there are finitely many measures, and, moreover, there is an effective upper bound on their number. In the body of the paper, we prove a more precise and general result (Theorem~\ref{mainthm1}); for instance, we show that the values of the measures are tightly constrained (they can only use primes dividing $m$).

There is a special class of measures called \defn{regular measures} (see \S \ref{ss:reg}); essentially, these are measures that take only non-zero values. Regular measures are important in the application to tensor categories since regularity is closely related to semi-simplicity. We obtain some results on regular measures too, such as the following. We say that the Fra\"iss\'e class $\fF$ is \defn{odd} if whenever $X \subset Y$ and $X \subset Z$ are inclusions of structures, the number of amalgamations of $Y$ and $Z$ over $X$ is odd.

\begin{theorem} \label{regthm}
Suppose the conditions of Theorem~\ref{mainthm} hold with $m$ odd. Then $\fF$ admits a regular measure if and only if $\fF$ is odd.
\end{theorem}

The above result is significant since it gives a completely combinatorial characterization of when regular measures exist. See \S \ref{ss:reg} for more precise results.

\subsection{Examples} \label{ss:ex}

We give a few examples of our main theorem; see \S \ref{ss:ex-details} for details.

(a) Let $\fF_s$ be the class of $s$-dimensional permutations, i.e., finite sets equipped with $s$ total orders. This class has received some attention in the literature, e.g., \cite{CameronPerms, BS}. Theorem~\ref{mainthm} applies to $\fF_s$ with $m=1$. A minimal marked structure has at most $2s+1$ points, and so the number of minimal marked structures is $O((2s+1)!^{s-1})$. The theorem thus gives $\Theta(\fF) \cong \bZ^{n(s)}$ with $n(s) = O(2^{(2s+1)!^{s-1}})$; the upper bound here is roughly $\exp(\exp(s^2 \log(s))$. We showed that $n(1)=4$ in \cite{repst} and $n(2)=37$ in \cite{homoperm}.

(b) Let $\fF_s$ be the class of finite sets equipped with a total order and an $s$-coloring. This Fra\"iss\'e class has also received some attention in the literature, e.g., \cite{Agarwal, Laflamme}. Theorem~\ref{mainthm} again applies with $m=1$. A minimal marked structure has at most three points, and so the number of minimal marked structures is $O(s^3)$. Thus the theorem shows that $\Theta(\fF) \cong \bZ^{n(s)}$ with $n(s) = O(2^{s^3})$. In \cite{colored}, we showed that $n(s)=(2s+2)^s$.

(c) Let $\fF$ be the class of finite boron trees, i.e., trees of valence at most three. See \cite[\S 2.6]{CameronBook} or \cite[\S 17]{repst} for details on this Fra\"iss\'e class. Let $\fF_s$ be defined similarly, but where the leaves now carry an $s$-coloring. Theorem~\ref{mainthm} applies to $\fF_s$ with $m=6$. A minimal marked structure has at most five leaves, and so the number of minimal marked structures is $O(s^5)$. The theorem thus shows $\Theta(\fF)[1/6] \cong \bZ[1/6]^{n(s)}$ with $n(s)=O(5^{s^5})$. We showed that $n(1)=2$ in \cite{repst}.

\subsection{Idea of proof}

To prove Theorem~\ref{mainthm}, we work with oligomorphic groups instead of Fra\"iss\'e classes. Thus instead of $\fF$ we have an oligomorphic group $(G, \Omega)$, and instead of $\Theta(\fF)$ we have $\Theta(G, \Omega)$. In broad strokes, the proof has two steps.

\textit{Step 1.} We show that $\Theta(G, \Omega)[1/m]$ is a finitely generated binomial ring. (See \S \ref{ss:binom} for the definition of binomial ring.) Finite generation follows from the condition Theorem~\ref{mainthm}(a); this was a fundamental observation of \cite{thesis}. We proved that certain related rings were binomial in \cite[\S 5]{repst}. With some effort, we transfer that result to $\Theta(G, \Omega)[1/m]$; this relies on the condition Theorem~\ref{mainthm}(b). The structure theorem for finitely generated binomial rings, due to Ekedahl, now shows
\begin{displaymath}
\Theta(G, \Omega)[1/m] \cong \bZ[1/m_1] \times \cdots \times \bZ[1/m_n]
\end{displaymath}
for some $m_1, \ldots, m_n \in \bZ$. In particular, $(G, \Omega)$ has $n$ complex measures. At this point, however, we have no control on the $m_i$'s or on $n$.

\textit{Step 2.} The ring $\Theta(G, \Omega)[1/m]$ comes with a natural generating set, the \emph{basic classes}. We show that the binomial coefficients of these elements have certain special properties. Combined with some elementary number theoretic results about binomial coefficients, this allows us to control the $m_i$'s and $n$.

\subsection{Outline}

In \S \ref{s:oligo} we study the universal ring $\Theta$, and prove results of a general nature. In \S \ref{s:main}, we prove our main theorems. Finally, in \S \ref{s:struct}, we translate our results into the language of finite relational structures.

\subsection{Notation}

We list some of the important notation:
\begin{description}[align=right,labelwidth=2.25cm,leftmargin=!,font=\normalfont]
\item[ $G$ ] an oligomorphic group
\item[ $\Omega$ ] the set on which $G$ acts
\item[ $G(A)$ ] the subgroup of $G$ fixing each element of $A \subset \Omega$
\item[ {$G[A]$} ] the subgroup of $G$ fixing $A$ setwise
\item[ $\sE$ ] a stabilizer class for $G$
\item[ $\fF$ ] a Fra\"iss\'e class
\item[ $\Theta(-)$ ] the ring carrying the universal measure
\end{description}
There are also a few named hypotheses we refer to:
\begin{description}[align=right,labelwidth=2.25cm,leftmargin=!,font=\normalfont]
\item[$\IND(m)$] the index of $G(A)$ in $G[A]$ divides a power of $m$, for all finite $A \subset \Omega$ (\S \ref{ss:binom})
\item[$\GEN(m)$] $\Theta_{\le 1}(G, \Omega)[1/m]$ is a finitely generated $\bZ[1/m]$-module (\S \ref{ss:filt})
\item[$\AUT(m)$] $\# \Aut(X)$ divides a power of $m$, for all $X \in \fF$ (\S \ref{ss:aut})
\item[$\FMM$] $\fF$ has finitely many minimal marked structures (\S \ref{ss:marked})
\end{description}

\subsection*{Acknowledgments}

We thank Nate Harman for helpful discussions.

\section{A filtered binomial ring} \label{s:oligo}

\subsection{Overview}

In \S \ref{s:oligo}, we review some material on oligomorphic groups and measures from \cite{repst}, and also prove several new results. Everything we do in this section has a specific reason behind it, but since these reasons are not always immediately evident, we attempt to describe the big picture here. We freely use terms defined below in this discussion.

Let $(G, \Omega)$ be an oligomorphic group. A measure is a rule assigning a number to each map of $G$-sets, and there is a ring $\Theta(G)$ that carries the universal measure for $G$. It turns out that the category of all $G$-sets is too big for our purposes. We prefer to work with what we call $\sE(\Omega)$-smooth $G$-sets, which, in the transitive case, are those $G$-sets that occur as an orbit on a power of $\Omega$; here $\sE(\Omega)$ is an example of a stabilizer class. There is a notion of measure based on this subclass of $G$-sets, and there is again a universal ring $\Theta(G, \Omega)$.

Given a transitive $\sE(\Omega)$-smooth $G$-set, one can consider the minimal power of $\Omega$ in which it appears. It turns out that this idea leads to a filtration on the ring $\Theta(G, \Omega)$. This filtration plays a key role in the proof of our main theorem (see Proposition~\ref{prop:pi}), and is the reason we work with $\Theta(G, \Omega)$ instead of $\Theta(G)$.

In \cite[\S 5]{repst}, we proved the important theorem that $\Theta(G)$ is a binomial ring, and binomiality will be a primary component in the proof of our main theorem. Unfortunately, the ring $\Theta(G, \Omega)$ is not binomial in general: the problem now is that the class of $\sE(\Omega)$-smooth $G$-sets is too small! We show that one can enlarge $\sE(\Omega)$ by a reasonably small amount to obtain a stabilizer class $\sE^+(\Omega)$ for which binomiality holds. An important problem is to estimate the difference between the ring $\Theta(G, \Omega)$ that we care about, and the ring $\Theta(G, \sE^+(\Omega))$ that we know to be binomial. The key result in this direction is Proposition~\ref{prop:large}.

To summarize: the main point of this section is to obtain a version of the $\Theta$ ring that carries the filtration mentioned above and is also a binomial ring. Achieving this requires a careful analysis of different stabilizer classes, and how the associated $\Theta$ rings compare.

\subsection{Oligomorphic groups}

An \defn{oligomorphic group} is a permutation group $(G, \Omega)$ such that $G$ has finitely many orbits on $\Omega^n$ for all $n$. Given an oligomorphic group and a finite subset $A$ of $\Omega$, let $G(A)$ be the subgroup of $G$ fixing each element of $A$. These subgroups form a neighborhood basis of the identity for a topology on $G$. This topology is Hausdorff, non-archimedean (open subgroups form a neighborhood basis of the identity), and Roelcke pre-compact (if $U$ and $V$ are open subgroups then $U \backslash G/V$ is a finite set). A topological group with these three properties is called \defn{pro-oligomorphic}. As the name suggests, such a group is essentially an inverse limit of oligomorphic groups.

Let $G$ be a pro-oligomorphic group. We say an action of $G$ on a set $X$ is \defn{smooth} (resp.\ \defn{finitary}) if every element of $X$ has open stabilizer in $G$ (resp.\ if $X$ has finitely many orbits). We use the term ``$G$-set'' to mean ``set equipped with a smooth and finitary action of $G$.'' We write $\bS(G)$ for the category of $G$-sets. This category is closed under binary products and fiber products \cite[Proposition~2.8]{repst}. See \cite[\S 2]{repst} for further background on oligomorphic groups and $G$-sets.

\begin{example} \label{ex:sym}
Let $\Omega=\{1,2,\ldots\}$ and let $\fS$ be the group of all permutations on $\Omega$. Then $(\fS, \Omega)$ is an oligomorphic permutation group. Let $\fS(n)$ be the subgroup of $\fS$ fixing each of $1, \ldots, n$. These subgroups form a neighborhood basis of the identity. Any open subgroup of $\fS$ is conjugate to $H \times \fS(n)$ for some $n$ and some subgroup $H$ of the finite symmetric group $\fS_n$ \cite[Proposition~14.1]{repst}.
\end{example}

\subsection{Algebraic closure}

Suppose $(G, \Omega)$ is oligomorphic. Let $A$ be a finite subset of $\Omega$. The \defn{algebraic closure} of $A$, denoted $\acl(A)$, is the set of elements of $\Omega$ that have finite orbit under $G(A)$. The action of $G(A)$ on $\Omega$ is oligomorphic, and so, in particular, has finitely many orbits; since $\acl(A)$ is the union of all finite orbits of $G(A)$, it is a finite set. The operation $\acl$ is idempotent, that is, $\acl(\acl(A))=\acl(A)$.

We will require the following result on algebraic closure when we estimate the difference between the $\Theta$ rings for $\sE(\Omega)$ and $\sE^+(\Omega)$ in the proof of Theorem~\ref{thm:theta-binom}. We write $G[A]$ for the group of elements $g \in G$ that fix $A$ setwise, i.e., $gA=A$. We note that $G[A]$ contains $G(A)$ as a finite index normal subgroup.

\begin{proposition} \label{prop:acl}
Let $U$ be a subgroup of $G$ that contains $G(A)$ with finite index, for some finite subset $A$ of $\Omega$. Letting $B=\acl(A)$, we have
\begin{displaymath}
G(B) \subset G(A) \subset U \subset G[B].
\end{displaymath}
\end{proposition}

\begin{proof}
Let $U=\bigsqcup_{i=1}^n h_i G(A)$ be the coset decomposition of $U$. Suppose $x \in B$ and $g \in U$. Then $Ugx=Ux=\bigcup_{i=1}^n h_i G(A) x$ is finite, since $G(A) x$ is finite. In particular, $G(A)gx$ is finite, and so $gx \in B$. We thus see that $U$ maps $B$ into itself, and so $U \subset G[B]$. Since $B$ contains $A$, we have $G(B) \subset G(A)$.
\end{proof}

\begin{corollary}
If $B$ is a finite subset of $\Omega$ that is algebraically closed then the normalizer of $G(B)$ in $G$ is $G[B]$.
\end{corollary}

\begin{proof}
Let $U$ be the normalizer of $G(B)$. This contains $G(B)$ as a subgroup, and so $U$ is open. Since $G(B)$ is a normal subgroup of $U$, it has finite index; indeed, $U$ is pro-oligomorphic, and so $U/G(B)=G(B) \backslash U/G(B)$ is finite. The proposition thus implies $U \subset G[B]$. Since $G[B]$ normalizes $G(B)$, we also have $G[B] \subset U$, and so the result follows.
\end{proof}

\begin{remark}
We give an example to illustrate some of the above concepts. Let $\Gamma$ be a disjoint union of copies of the complete graph $K_n$ (for $n$ fixed), and let $G=\fS_n \wr \fS$ be its automorphism group. The group $G$ acts oligomorphically on the vertex set $\Omega$ of $\Gamma$. Let $A$ be a subset of the first component of $\Gamma$ of cardinality $r>0$. Then
\begin{displaymath}
G(A) \cong \fS_{n-r} \times (\fS_n \wr \fS(1)), \qquad
G[A] \cong (\fS_r \times \fS_{n-r}) \times (\fS_n \wr \fS(1))
\end{displaymath}
where $\fS(1) \subset \fS$ is the stabilizer of~1 (as in Example~\ref{ex:sym}). In this case, $G[A]$ is the normalizer of $G(A)$ in $G$. The algebraic closure $B$ of $A$ consists of all vertices in the first component. The groups $G(B)$ and $G[B]$ are as above, but with $r=n$. We also note that the \defn{definable closure} of $A$, which is defined to be the set of fixed points of $G(A)$, is $A$ if $r \ne n-1$, and coincides with the algebraic closure $B$ if $r=n-1$.
\end{remark}

\subsection{Stabilizer classes} \label{ss:stab}

Let $G$ be a pro-oligomorphic group. We will typically not want to work with all $G$-sets. The classes of sets we will use can be conveniently captured with using the following notion, introduced in \cite[\S 2.6]{repst}.

\begin{definition}
A \defn{stabilizer class} in $G$ is a collection $\sE$ of open subgroups of $G$ satisfying the following conditions:
\begin{enumerate}
\item $\sE$ contains $G$.
\item $\sE$ forms a neighborhood basis for the identity of $G$.
\item $\sE$ is closed under conjugation, i.e., $U \in \sE$ and $g \in G$ implies $gUg^{-1} \in \sE$.
\item $\sE$ is closed under finite intersections, i.e., $U,V \in \sE$ implies $U \cap V \in \sE$. \qedhere
\end{enumerate}
\end{definition}

Let $\sE$ be a stabilizer class. We say that a $G$-set $X$ is \defn{$\sE$-smooth} if the stabilizer of any element of $X$ belongs to $\sE$. We let $\bS(G, \sE)$ be the full subcategory of $\bS(G)$ spanned by the $\sE$-smooth sets. The category $\bS(G, \sE)$ is also closed under binary products and fiber products. An intrinsic characterization of these categories is given in \cite{pregalois}.

Suppose $(G, \Omega)$ is oligomorphic. We let $\sE(\Omega)$ be the collection of all open subgroups of the form $G(A)$, with $A$ a finite subset of $\Omega$; alternatively, this can be described as the set of subgroups that occur as the stabilizer of some point on $\Omega^n$. This is a stabilizer class; indeed, axioms~(a) and~(b) are evident, while~(c) follows from $gG(A)g^{-1}=G(gA)$, and~(d) from $G(A) \cap G(B)=G(A \cup B)$. A transitive $G$-set is $\sE(\Omega)$-smooth if and only if it occurs as an orbit on $\Omega^n$ for some $n$. We write $\bS(G, \Omega)$ in place of $\bS(G, \sE(\Omega))$.

\begin{example}
For the oligomorphic group $(\fS, \Omega)$ from Example~\ref{ex:sym}, $\sE(\Omega)$ consists of those subgroups conjugate to some $\fS(n)$. 
\end{example}

Suppose $\sE$ is a stabilizer class. We define $\sE^+$ to be the set of open subgroups of $G$ that contain some member of $\sE$ with finite index.

\begin{proposition}
The collection $\sE^+$ is a stabilizer class.
\end{proposition}

\begin{proof}
It is clear that $\sE^+$ satisfies conditions (a), (b), and (c). We now verify (d). Thus suppose that $U$ and $V$ belong to $\sE^+$. By definition, there are $U_0, V_0 \in \sE$ such that $U$ contains $U_0$ with finite index, and similarly for $V$. Thus $U \cap V$ contains $U_0 \cap V_0$ with finite index. As $U_0 \cap V_0$ belongs to $\sE$, it follows that $U \cap V$ belongs to $\sE^+$, which completes the proof.
\end{proof}

The stabilizer class $\sE^+$ has a natural interpretation from the point of view of $G$-sets: a transitive $G$-set $X$ is $\sE^+$-smooth if and only if there is a finite-to-one map $Y \to X$ with $Y$ a transitive $\sE$-smooth set. In particular, we see that $\bS(G, \sE^+)$ is closed under quotients by finite groups, which is an important property in the analysis of measures (\S \ref{ss:binom}). If $G$ is oligomorphic and $\sE=\sE(\Omega)$ then Proposition~\ref{prop:acl} gives a stronger result: a transitive $G$-set $X$ is $\sE^+$-smooth if and only if $X=Y/\Gamma$ for some transitive $\sE$-smooth $G$-set $Y$ and subgroup $\Gamma$ of $\Aut_G(Y)$.

\subsection{Measures} \label{ss:meas}

Let $G$ be a pro-oligomorphic group and let $\sE$ be a stabilizer class for $G$. The following definition introduces the central objects of study in this paper:

\begin{definition} \label{defn:meas}
A \defn{measure} for $(G, \sE)$ valued in a commutative ring $k$ is a rule that assigns to each map $f \colon Y \to X$ of $\sE$-smooth $G$-sets, with $X$ transitive, a quantity $\mu(f)$ in $k$ such that the following conditions hold:
\begin{enumerate}
\item If $f$ is an isomorphism then $\mu(f)=1$.
\item Suppose $f \colon Y \to X$ is as above and $Y=Y_1 \sqcup Y_2$. Letting $f_i$ be the restriction of $f$ to $Y_i$, we have $\mu(f) = \mu(f_1) + \mu(f_2)$.
\item For composable maps $g \colon Z \to Y$ and $f \colon Y \to X$ of transitive $\sE$-smooth $G$-sets, we have $\mu(f \circ g)=\mu(f) \cdot \mu(g)$.
\item Let $f \colon Y \to X$ and $g \colon X' \to X$ be maps of $\sE$-smooth $G$-sets, with $X$ and $X'$ transitive, and let $f' \colon Y' \to X'$ be the base change of $f$. Then $\mu(f) = \mu(f')$.
\end{enumerate}
This definition is based on \cite[Definition~3.12]{repst}. It is equivalent to \cite[Definition~3.1]{repst} by \cite[Proposition~3.13]{repst}.
\end{definition}

Define a ring $\Theta(G, \sE)$ as follows. For each map $f \colon Y \to X$ of $\sE$-smooth $G$-sets, with $X$ transitive, there is a class $[f]$ in $\Theta(G, \sE)$. These classes satisfy analogs of (a)--(d) above. To be more precise, $\Theta(G, \sE)=P/I$, where $P$ is the polynomial ring in the symbols $[f]$ and $I$ is the ideal generated by the relations (a)--(d). There is a measure $\mu_{\rm univ}$ valued in $\Theta(G, \sE)$ defined by $\mu_{\rm univ}(f)=[f]$. This measure is universal, in the sense that if $\mu$ is a measure valued in a ring $k$ then there is a unique ring homomorphism $\phi \colon \Theta(G, \sE) \to k$ such that $\mu=\phi \circ \mu_{\rm univ}$. Thus the problem of determining all measures amounts to computing the ring $\Theta(G, \sE)$.

There are two shorthands we employ. When $\sE$ is the collection of all open subgroups of $G$, we write $\Theta(G)$ instead of $\Theta(G, \sE)$, and simply say ``measure for $G$.'' When $(G, \Omega)$ is oligomorphic, we write $\Theta(G, \Omega)$ instead of $\Theta(G, \sE(\Omega))$, and say ``measure for $(G, \Omega)$.''

\begin{example}
Consider the infinite symmetric group $(\fS, \Omega)$ from Example~\ref{ex:sym}. Then $\Theta(\fS, \Omega)=\bZ[t]$ while $\Theta(\fS)=\bZ\langle t \rangle$ is the ring of integer-valued polynomials (see \cite[Proposition~14.15]{repst} and \cite[Theorem~14.4]{repst}). In each case, the class $[\Omega \to \bone]$ corresponds to $t$, where $\bone$ is the one-point $\fS$-set.

Let $\Omega^{[n]}$ be the complement of the large diagonal in $\Omega^n$ and let $\Omega^{(n)}$ denote set of $n$-element subsets of $\Omega$. Both these sets are smooth, but $\Omega^{(n)}$ is not $\sE(\Omega)$-smooth. Indeed, $\Omega^{(n)}$ is the quotient of $\Omega^{[n]}$ by its natural $\fS_n$-action; the stabilizers on $\Omega^{(n)}$ are conjugate to $\fS_n \times \fS(n)$, which is not present in the class $\sE(\Omega)$.  The class $[\Omega^{[n]} \to \bone]$ in $\Theta(\fS)$ is equal to $t(t-1)\dots(t-n+1)$, while $[\Omega^{(n)} \to \bone]$ equals $\binom{t}{n}$, the quotient of the former by $\# \fS_n$. See Proposition~\ref{prop:bcoeff} for details of the computations.
\end{example}

\subsection{Change of stabilizer class} \label{ss:chgstb}

Let $G$ be a pro-oligomorphic group and let $\sE \subset \sF$ be stabilizer classes for $G$. One easily sees that a measure for $(G, \sF)$ restricts to a measure for $(G, \sE)$. Dually, there is a natural ring homomorphism
\begin{displaymath}
\Theta(G, \sE) \to \Theta(G, \sF).
\end{displaymath}
We say that $\sE$ is \defn{large} in $\sF$ if each subgroup in $\sF$ contains a subgroup in $\sE$ with finite index. In this case, the above map is an isomorphism after tensoring up to $\bQ$. This statement was asserted in \cite[\S 3.4(d)]{repst}, though the proof was omitted. We will require a stronger statement; we include a complete proof, as this result is crucial to this paper.

Let $S$ be an arbitrary set of prime numbers. We say that $\sE \subset \sF$ is \defn{$S$-large} if for every $V \in \sF$ there is $U \in \sE$ such that $V$ contains $U$ with finite index, and every prime factor of $[V:U]$ belongs to $S$. We write $M[1/S]$ for the localization of a $\bZ$-module $M$ at the multicative set generated by $S$. The following is the result we require; it recovers \cite[\S 3.4(d)]{repst} when $S$ is the set of all prime numbers.

\begin{proposition} \label{prop:large}
If $\sE \subset \sF$ be $S$-large then $\Theta(G, \sE)[1/S] = \Theta(G, \sF)[1/S]$.
\end{proposition}

The proof will take the remainder of \S \ref{ss:chgstb}. For a map of sets $f \colon Y \to X$, we write $\delta(f)=n$ if $f$ is everywhere $n$-to-1, i.e., $f^{-1}(x)$ has cardinality $n$ for all $x \in X$ (and $X$ is non-empty). If $f$ and $g$ are composable and $\delta(f)$ and $\delta(g)$ are defined then $\delta(fg)=\delta(f) \delta(g)$. Also, if $\delta(f)$ is defined and $f'$ is a base change of $f$ then $\delta(f')=\delta(f)$. We say that a positive integer is \defn{$S$-smooth} if all of its prime factors belong to $S$, and we say that $f$ is an \defn{$S$-cover} if $\delta(f)$ is defined and $S$-smooth. If $f$ is a map of $\sE$-smooth $G$-sets and $\delta(f)$ is defined then $[f]=\delta(f)$ in $\Theta(G, \sE)$ by \cite[\S 3.4(b)]{repst}.

\begin{lemma}
Let $X$ be an $\sF$-smooth $G$-set. Then there is an $S$-cover $Y \to X$ where $Y$ is an $\sE$-smooth $G$-set (and the map is $G$-equivariant).
\end{lemma}

\begin{proof}
First suppose $X=G/U$ is transitive, with $U \in \sF$. Let $V \in \sE$ be a subgroup contained in $U$ such that $[U:V]$ is finite and $S$-smooth. Then we can take $Y=G/V$. We now treat the general case. Let $X=X_1 \sqcup \cdots \sqcup X_n$ be the decomposition of $X$ into orbits. For each $i$, choose an $S$-cover $Y_i \to X_i$ where $Y_i$ is $\sE$-smooth. Let $d_i=\delta(Y_i \to X_i)$, and let $d$ be the least common multiple of the $d_i$'s. We can then take $Y=Y_1^{\sqcup d/d_1} \sqcup \cdots \sqcup Y_n^{\sqcup d/d_n}$. The natural map $Y \to X$ is an $S$-cover with $\delta(Y \to X)=d$.
\end{proof}

We now construct a partially defined measure $\mu_1$ for $(G, \sF)$ valued in $\Theta(G, \sE)[1/S]$. This measure will be defined for maps $f \colon Y \to X$ where $X$ is $\sE$-smooth and transitive, and $Y$ is $\sF$-smooth. To define it, we choose an $S$-cover $h \colon Y_1 \to Y$, with $Y_1$ an $\sE$-smooth $G$-set, and put $\mu_1(f)=\delta(h)^{-1} [fh]$.  In what follows $[-]$ always means the class of a map in $\Theta(G, \sE)$. The main properties of this construction are summarized below.

\begin{lemma} \label{lem:large-2}
The construction $\mu_1$ satisfies the following properties. In what follows, let $f \colon Y \to X$ be a map where $X$ is transitive and $\sE$-smooth, and $Y$ is $\sF$-smooth.
\begin{enumerate}
\item $\mu_1$ is well-defined.
\item $\mu_1$ is additive, i.e., it satisfies a version of Definition~\ref{defn:meas}(b).
\item $\mu_1$ is compatible with finite covers: if $g \colon Z \to Y$ is a map where $Z$ is $\sF$-smooth and $\delta(g)$ is defined then $\mu_1(fg)=\delta(g) \mu_1(f)$.
\item $\mu_1$ is compatible with composition: if $Y$ is transitive and $\sE$-smooth and $g \colon Z \to Y$ is a map with $Z$ an $\sF$-smooth $G$-set then $\mu_1(fg)=\mu_1(f) \mu_1(g)$.
\item $\mu_1$ is compatible with base change: if $X' \to X$ is a map of transitive $\sE$-smooth $G$-sets and $f' \colon Y' \to X'$ is the base change of $f$ then $\mu_1(f)=\mu_1(f')$.
\end{enumerate}
\end{lemma}

\begin{proof}
(a) Suppose $h' \colon Y_2 \to Y$ is a second $S$-cover with $Y_2$ an $\sE$-smooth $G$-set. Let $Y_3=Y_1 \times_Y Y_2$; this is a union of orbits on $Y_1 \times Y_2$, and thus $\sE$-smooth. We have $\delta(Y_3 \to Y_1)=\delta(Y_2 \to Y)$, and so the former is $S$-smooth. We have
\begin{displaymath}
\delta(h)^{-1} \cdot [fh] = \delta(Y_3 \to Y)^{-1} \cdot [Y_3 \to X]
\end{displaymath}
in $\Theta(G,\sE)[1/S]$. A similar computation shows that $\delta(h')^{-1} \cdot [fh']$ is equal to the same, and so the result follows.

(b) This is clear from the definition.

(c) Choose a cartesian diagram
\begin{displaymath}
\xymatrix{
Z_1 \ar[r]^{h'} \ar[d]_{g_1} & Z \ar[d]^g \\
Y_1 \ar[r]^h & Y }
\end{displaymath}
where $h$ is an $S$-cover and $Y_1$ is $\sE$-smooth. Note that $\delta(g_1)=\delta(g)$ and $\delta(h)=\delta(h')$. Also choose an $S$-cover $h'' \colon Z_2 \to Z_1$ where $Z_2$ is $\sE$-smooth. Note that $h' \circ h''$ is also an $S$-cover. We have
\begin{displaymath}
\mu_1(fg) = \delta(h'h'')^{-1} [fgh'h''] = \delta(h)^{-1} \delta(h'')^{-1} [fhg_1h''] = \delta(h)^{-1} \delta(g) [fh] = \delta(g) \mu_1(f).
\end{displaymath}
In the first step, we used the definition of $\mu_1$; in the second, the commutativity of the diagram and properties of $\delta$; in the third, the factorization
\begin{displaymath}
fhg_1h''=(fh) \circ (g_1h'')
\end{displaymath}
in $\bS(G, \sE)$, together with the computation
\begin{displaymath}
[g_1h'']=\delta(g_1h'') = \delta(g) \delta(h'');
\end{displaymath}
and in the fourth step, the definition of $\mu_1$ again. The result follows.

(d) Let $h \colon Z_1 \to Z$ be an $S$-cover with $Z_1$ an $\sE$-smooth $G$-set. Then
\begin{displaymath}
\mu_1(fg) = \delta(h)^{-1} [fgh] = \delta(h)^{-1} [f] [gh] = \mu_1(f) \mu_1(g).
\end{displaymath}
In the first step, we used the definition of $\mu_1$; in the second, the factorization $fgh=f \circ (gh)$ in $\bS(G, \sE)$; and in the third, the definition of $\mu_1$ again.

(e) It suffices to treat the case where $Y$ is transitive. Let $h \colon Y_1 \to Y$ an $S$-cover with $Y_1$ transitive and $\sE$-smooth, and let $h' \colon Y_1' \to Y'$ be its base change. Note that $h'$ is an $S$-cover and $\delta(h')=\delta(h)$. We have
\begin{displaymath}
\mu_1(f') = \delta(h)^{-1} [f'h'] = \delta(h)^{-1} [fh] = \mu_1(f).
\end{displaymath}
In the first step, we used the definition of $\mu_1$; in the second, that $f'h'$ is a base change of $fh$ in the category $\bS(G, \sE)$; and in the third, the definition of $\mu_1$ again.
\end{proof}

We now define an actual measure $\mu$ for $(G, \sF)$ valued in $\Theta(G, \sE)[1/S]$. Thus suppose $f \colon Y \to X$ is given, where $X$ and $Y$ are $\sF$-smooth and $X$ is transitive. Choose a map $X_1 \to X$ with $X_1$ transitive and $\sE$-smooth, and let $f_1 \colon Y_1 \to X_1$ be the base change of $f$. We define $\mu(f)=\mu_1(f_1)$. It is clear that $\mu$ satisfies Definition~\ref{defn:meas}(a,b).

\begin{lemma} \label{lem:large-4}
$\mu$ is well-defined and compatible with base change, i.e., Definition~\ref{defn:meas}(d) holds.
\end{lemma}

\begin{proof}
Suppose $X_2 \to X$ is a second map with $X_2$ transitive and $\sE$-smooth, and let $X_3$ be any $G$-orbit on $X_1 \times_X X_2$. Then, in the obvious notation, we have $\mu_1(f_1)=\mu_1(f_3)=\mu_1(f_2)$, where in both steps we use Lemma~\ref{lem:large-2}(e). Thus shows that $\mu$ is well-defined.

Now suppose that $X' \to X$ is a map of transitive $\sF$-smooth $G$-sets, and let $f' \colon Y' \to X'$ be the base change of $f$. Choose a commutative diagram
\begin{displaymath}
\xymatrix{
X'_1 \ar[r] \ar[d] & X_1 \ar[d] \\
X' \ar[r] & X }
\end{displaymath}
where $X_1$ and $X_1'$ are transitive and $\sE$-smooth; for example, first choose $X_1$, and then choose $X_1'$ by taking an appropriate cover of an orbit on the fiber product. In the obvious notation, we have $\mu(f)=\mu_1(f_1)$ and $\mu(f')=\mu_1(f'_1)$, by definition of $\mu$; we also have $\mu_1(f_1)=\mu_1(f_1')$ by Lemma~\ref{lem:large-2}(e). Thus the result follows.
\end{proof}

\begin{lemma} \label{lem:large-5}
$\mu$ is compatible with composition, i.e., Definition~\ref{defn:meas}(c) holds.
\end{lemma}

\begin{proof}
Let $f \colon Y \to X$ and $g \colon Z \to Y$ be maps of transitive $\sF$-smooth $G$-sets. First suppose that $X$ is $\sE$-smooth. Consider a cartesian square
\begin{displaymath}
\xymatrix{
Z_1 \ar[r]^{h'} \ar[d]_{g_1} & Z \ar[d]^g \\
Y_1 \ar[r]^h & Y }
\end{displaymath}
where $h$ is an $S$-cover with $Y_1$ transitive and $\sE$-smooth. We have
\begin{align*}
\delta(h') \mu(fg) &= \delta(h') \mu_1(fg) = \mu_1(fgh') = \mu_1(fhg_1) \\
& =  \mu_1(hf) \mu_1(g_1) = \delta(h) \mu_1(f) \mu_1(g_1) = \delta(h) \mu(f) \mu(g).
\end{align*}
In the first step, we applied the definition of $\mu(fg)$; in the second, we used Lemma~\ref{lem:large-2}(c); in the third, we used commutativity of the diagram; in the fourth, we used Lemma~\ref{lem:large-2}(d); in the fifth, we used Lemma~\ref{lem:large-2}(c) again; and in the final step, we applied the definition of $\mu$ and used Lemma~\ref{lem:large-4}. It follows that $\mu(fg)=\mu(f) \mu(g)$ in $\Theta(\sE)[S^{-1}]$, as required.

We now treat the general case. Let $X_1 \to X$ be a map with $X_1$ transitive and $\sE$-smooth. Let $f_1 \colon Y_1 \to X_1$ and $g_1 \colon Z_1 \to Y_1$ be the base changes of $f$ and $g$. Let $Y_{1,1}, \ldots, Y_{1,n}$ be the orbits on $Y_1$, let $Z_{1,i}$ be the inverse image of $Y_{1,i}$ in $Z_1$, and let $f_{1,i} \colon Y_{1,i} \to X_1$ and $g_{1,i} \colon Z_{1,i} \to Y_{1,i}$ be the restrictions of $f_1$ and $g_1$. Note that $g_{1,i}$ is the base change of $g$ along $Y_{1,i} \to Y$, and so $\mu(g_{1,i})=\mu(g)$ for each $i$ by Lemma~\ref{lem:large-4}. We have
\begin{displaymath}
\mu(fg)=\mu(f_1 g_1) = \sum_{i=1}^n \mu(f_{1,i}) \mu(g_{1,i}) = \mu(f_1) \mu(g) = \mu(f) \mu(g).
\end{displaymath}
In the first step, we used Lemma~\ref{lem:large-4}; in the second, the additivity of $\mu$; in the third, that $\mu(g_{1,i})=\mu(g)$ and again the additivity of $\mu$; and in the final step, Lemma~\ref{lem:large-4} again. This completes the proof.
\end{proof}

We have thus constructed the measure $\mu$. We can now finish up.

\begin{proof}[Proof of Proposition~\ref{prop:large}]
As discussed before Proposition~\ref{prop:large}, there is a natural ring homomorphism
\begin{displaymath}
i \colon \Theta(G, \sE)[1/S] \to \Theta(G, \sF)[1/S].
\end{displaymath}
The measure $\mu$ furnishes us a ring homomorphism $j$ in the opposite direction. It is clear from the construction of $\mu$ that $j \circ i$ is the identity, and so $i$ is injective. To complete the proof, we show that $i$ is surjective.

In what follows, we write $\{ - \}$ for classes in $\Theta(G, \sF)$. Let $\{f\}$ be given, where $f \colon Y \to X$ is a map of $\sF$-smooth $G$-sets with $X$ transitive. We must show that $\{f\}$ belongs to the image of $i$. Let $X' \to X$ be a map where $X'$ is $\sE$-smooth and transitive, and let $f' \colon Y' \to X'$ be the base change of $f$. Next, choose an $S$-cover $h \colon Y_1 \to Y'$ where $Y_1$ is $\sE$-smooth. Then
\begin{displaymath}
i([f'h]) = \{f'h\} = \delta(h) \{f'\} = \delta(h) \{f\}.
\end{displaymath}
Thus $\{f\}$ is the image of the element $\delta(h)^{-1} [f'h]$, and so the result follows.
\end{proof}

\subsection{Binomiality} \label{ss:binom}

A ring $R$ is a \defn{binomial ring} if it is $\bZ$-torsion free and for every $x \in R$ and $n \in \bN$ the element $\binom{x}{n}$ of $R \otimes \bQ$ belongs to $R$. To state our results on binomiality of $\Theta$ rings, we make use of the following condition.
\begin{description}[align=right,labelwidth=1.75cm,leftmargin=!,font=\normalfont]
\item[$\IND(m)$] For finite $A \subset \Omega$, the index of $G(A)$ in $G[A]$ divides a power of $m$.
\end{description}
One more piece of notation: for a $G$-set $X$, let $X^{[n]}$ be the subset of $X^n$ with distinct coordinates, and let $X^{(n)}=X^{[n]}/\fS_n$ be the set of $n$-element subsets of $X$. We are now ready to give our main result on binomiality. In what follows, $(G, \Omega)$ is an oligomorphic group.

\begin{theorem} \label{thm:theta-binom}
Suppose $\IND(m)$ holds. Then $\Theta(G, \Omega)[1/m]$ is a binomial ring.
\end{theorem}

\begin{proof}
Let $\sE=\sE(\Omega)$. Recall that the category $\bS(G, \sE^+)$ is closed under quotients by finite groups. In particular, if $X$ is an $\sE^+$-smooth $G$-set then so is $X^{(n)}$, for any $n \ge 0$. It follows from \cite[Theorem~5.1]{repst} and \cite[Remark~5.7]{repst} that $\Theta(G, \sE^+)$ is binomial. Hence $\Theta(G, \sE^+)[1/m]$ is also binomial \cite[Proposition~5.1]{Elliott}.

Let $S$ be the set of prime divisors of $m$. We claim that $\sE \subset \sE^+$ is $S$-large; this will establish the theorem by Proposition~\ref{prop:large}. Let $U \in \sE^+$ be given. By definition, there is some finite subset $A$ of $\Omega$ such that $U$ contains $G(A)$ with finite index. Let $B=\acl(A)$ be the algebraic closure of $A$. Then we have containments
\begin{displaymath}
G(B) \subset G(A) \subset U \subset G[B]
\end{displaymath}
by Proposition~\ref{prop:acl}. Thus $[U:G(A)]$ divides $[G[B]:G(B)]$, and therefore divides some power of $m$ by $\IND(m)$. This proves the claim.
\end{proof}

We now explain how to compute binomial coefficients in $\Theta$ rings. In what follows, $G$ is pro-oligomorphic. Let $f \colon Y \to X$ be a map of $G$-sets, with $X$ transitive. We let $Y^n_{/X}$ be the $n$-fold fiber product of $Y$ with itself over $X$, we let $Y^{[n]}_{/X}$ be the subset of $Y^n_{/X}$ where the large diagonal has been removed, and we let $Y^{(n)}_{/X}$ be the quotient of $Y^{[n]}_{/X}$ by $\fS_n$. We note that each of these sets admits a natural map to $X$.

\begin{proposition} \label{prop:bcoeff}
Let $\sE$ be a stabilizer class in $G$. Let $f \colon Y \to X$ be a map of $\sE$-smooth $G$-sets, with $X$ transitive, let $x=[f]$ be its class in $\Theta(G, \sE)$, and let $n \ge 0$ be an integer. Then
\begin{displaymath}
x(x-1)\cdots(x-n+1) = [Y^{[n]}_{/X} \to X]
\end{displaymath}
holds in $\Theta(G, \sE)$. Furthermore, there exist $G$-orbits $Z_1, \ldots, Z_r$ on $Y^{[n]}_{/X}$ and positive integers $d_1, \ldots, d_r$ such that $d_i$ divides $\# \Aut_G(Z_i)$ and
\begin{displaymath}
\binom{x}{n} = \sum_{i=1}^r d_i^{-1} \cdot [Z_i \to X]
\end{displaymath}
holds in $\Theta(G, \sE) \otimes \bQ$.
\end{proposition}

\begin{proof}
First suppose that $h_1 \colon Y_1 \to X$ and $h_2 \colon Y_2 \to X$ are maps of $\sE$-smooth $G$-sets, and let $h \colon Y_1 \times_X Y_2 \to X$ be the natural map. We claim that $[h_1][h_2]=[h]$. To see this, it suffices to treat the case where $Y_1$ and $Y_2$ are transitive. Let $h_2' \colon Y_1 \times_X Y_2 \to Y_1$ be the projection map. Then $[h_2]=[h_2']$, and $[h]=[h_2'][h_1]$. This establishes the claim.

From the above paragraph, we see that $x^n = [Y^n_{/X} \to X]$. The first formula now follows from some standard combinatorics; see \cite[Proposition~5.15]{repst} for a similar argument. Since the symmetric group $\fS_n$ acts freely on $Y^{[n]}_{/X}$, we have $[Y^{(n)}_{/X}]=\binom{x}{n}$ by \cite[\S 3.4(b)]{repst}. The set $Y^{[n]}_{/X}$ decomposes into $G$-orbits, and these orbits are permuted by $\fS_n$. Let $Z_1, \ldots, Z_r$ be a system of representatives for this $\fS_n$-action, i.e., each $Z_i$ is a $G$-orbit on $Y^{[n]}_{/X}$, and each $G$-orbit belongs to the $\fS_n$-orbit of exactly one $Z_i$. Let $\Gamma_i \subset \Aut_G(Z_i)$ be the stabilizer in $\fS_n$ of $Z_i$. Then $Y^{(n)}_{/X}$ is the disjoint union of the $G$-sets $Z_i/\Gamma_i$. Thus the final formula follows, with $d_i=\# \Gamma_i$.
\end{proof}

\subsection{The filtration on $\Theta$} \label{ss:filt}

Let $(G, \Omega)$ be oligomorphic. We say that a map $f \colon Y \to X$ of transitive $G$-sets is \defn{basic} if $Y$ is an orbit on $\Omega^{\ell+1}$, $X$ is an orbit on $\Omega^{\ell}$, and $f$ is the restriction of the projection onto the first $\ell$ factors; we also say that the class $[f] \in \Theta(G, \Omega)$ is \defn{basic}. We let $\Omega_{\le 1}(G, \Omega)$ be the additive subgroup of $\Theta(G, \Omega)$ generated by basic classes; more generally, we let $\Omega_{\le n}(G, \Omega)$ be the additive group generated by all $n$-fold products of basic classes.

\begin{proposition} \label{prop:filt}
We have the following:
\begin{enumerate}
\item Let $p \colon \Omega^{n+\ell} \to \Omega^{\ell}$ be a projection onto some set of coordinates, let $Y$ be a $G$-orbit on $\Omega^{n+\ell}$, let $X=p(Y)$, and let $f \colon Y \to X$ be the restriction of $p$. Then $[f]$ is a product of $n$ basic classes, and thus belongs to $\Theta_{\le n}(G, \Omega)$.
\item Let $f \colon Y \to X$ be a map of transitive $\sE(\Omega)$-smooth $G$-sets. Then $[f]$ is a product of basic classes.
\item The ring $\Theta(G, \Omega)$ is generated by $\Theta_{\le 1}(G, \Omega)$; in other words, $\Theta(G, \Omega)$ is the union of the $\Theta_{\le n}(G, \Omega)$.
\end{enumerate}
\end{proposition}

\begin{proof}
(a) We may as well assume that $p$ projects onto the first $\ell$ coordinates. For $0 \le i \le n$, let $Y_i$ be the image of $Y$ in $\Omega^{\ell+i}$ under the projection onto the first $\ell+i$ coordinates. For $1 \le i \le n$, the projection map $f_i \colon Y_i \to Y_{i-1}$ is basic. As $f=f_1 \circ \cdots \circ f_n$, we have $[f]=[f_1] \cdots [f_{\ell}]$, as required.

(b) Suppose that $Y$ is an orbit on $\Omega^n$ and $X$ is an orbit on $\Omega^{\ell}$. Let $\Gamma \subset \Omega^{n+\ell}$ be the graph of $f$, and let $f' \colon \Gamma \to X$ be the projection map. Then $f'$ is isomorphic to $f$, and so $[f]=[f']$. By (a), $f'$ is a product of $n$ basic classes.

(c) If $f \colon Y \to X$ is a map of transitive $\sE$-smooth $G$-sets then $[f]$ belongs to $\bigcup_{n \ge 1} \Theta_{\le n}(G, \Omega)$ by (b). Since such classes generate $\Theta(G, \Omega)$, the result follows.
\end{proof}

We will make use of the following condition.
\begin{itemize}[leftmargin=2cm]
\item[$\GEN(m)$:] $\Theta_{\le 1}(G, \Omega)[1/m]$ is a finitely generated $\bZ[1/m]$-module.
\end{itemize}
We will see (Example~\ref{ex:FMM-vs-Fin}) that this is weaker than the condition in Theorem~\ref{mainthm}(b). We note a simple corollary of the proposition:

\begin{corollary} \label{cor:theta-fg}
If $\GEN(m)$ holds then $\Theta(G, \Omega)[1/m]$ is a finitely generated ring.
\end{corollary}

\section{The main results} \label{s:main}

\subsection{Overview}

Fix an oligomorphic group $(G, \Omega)$ and a positive integer $m$ throughout \S \ref{s:main}. Assuming $\IND(m)$ and $\GEN(m)$, the results from \S \ref{s:oligo} show that $\Theta(G, \Omega)[1/m]$ is a finitely generated binomial ring. The structure theorem for such rings thus tells us
\begin{displaymath}
\Theta(G, \Omega)[1/m] = \bZ[1/m_1] \times \cdots \times \bZ[1/m_n]
\end{displaymath}
for some integers $m_1, \ldots, m_n$. This is already a strong result, since it tells us that there are finitely many $\bC$-valued measures for $(G, \Omega)$; in fact, there are exactly $n$ of them. However, we would like to be able to say something about the $m_i$'s and $n$. This is what our main theorem accomplishes: we show $m_i=m$ for all $i$, and give an upper bound on $n$.

The key idea is that if $x$ is a basic class in $\Theta(G, \Omega)$ then we have a meaningful description of $\binom{x}{n}$ from Proposition~\ref{prop:bcoeff}, and this leads to two important properties of these binomial coefficients (Proposition~\ref{prop:pi}). So $\Omega(G, \Omega)[1/m]$ has more structure than just a binomial ring: it comes with a generating set whose binomial coefficients are constrained. Combining this observation with some elementary number theoretic results about binomial coefficients (see \S \ref{ss:binom3}) leads to our desired result.

\subsection{Binomial coefficients} \label{ss:binom2}

The following proposition gives two important properties of binomial coefficients of basic classes in $\Theta(G, \Omega)$.

\begin{proposition} \label{prop:pi}
Suppose $\IND(m)$ holds. Let $x \in \Theta(G, \Omega)$ be a basic class, and let $n$ be a positive integer.
\begin{enumerate}
\item The element $\binom{x}{n}$ belongs to $\Theta_{\le n}(G, \Omega)[1/m]$.
\item The element $x$ divides the element $\binom{x}{n}$ in the ring $\Theta(G, \Omega)[1/m]$.
\end{enumerate}
\end{proposition}

\begin{proof}
Note that $\Theta(G, \Omega)[1/m]$ is a binomial ring (Theorem~\ref{thm:theta-binom}). Suppose $x=[Y \to X]$, where $Y$ is an orbit on $\Omega^{\ell+1}$, $X$ is the projection of $Y$ to $\Omega^{\ell}$, and $Y \to X$ is the projection map. By Proposition~\ref{prop:bcoeff}, we have
\begin{displaymath}
\binom{x}{n} = \sum_{i=1}^r d_i^{-1} \cdot [Z_i \to X],
\end{displaymath}
in $\Theta(G, \Omega) \otimes \bQ$, using the same notation as there. Both sides of this equation belong to $\Theta(G, \Omega)[1/m]$ by our assumptions. Since this ring is torsion-free (as it is binomial), the above equality actually holds in this ring.

Now, $Y^{[n]}_{/X}$ is naturally a $G$-subset of $\Omega^{\ell+n}$, and the map to $X$ is simply the projection onto the first $\ell$ coordinates. Thus $[Z_i \to X]$ belongs to $\Theta_{\le n}(G, \Omega)$ by Proposition~\ref{prop:filt}(a), which proves statement (a). The image of $Z_i$ in $\Omega^{\ell+1}$ under the projection map on the first $\ell+1$ coordinates is simply $Y$. We thus have
\begin{displaymath}
[Z_i \to X] = [Z_i \to Y] \cdot [Y \to X],
\end{displaymath}
which shows that each $[Z_i \to X]$ is divisible by $x$. Hence (b) follows.
\end{proof}

\subsection{More binomial coefficients} \label{ss:binom3}

We collect here two lemmas on binomial coefficients of rational numbers that we will need to prove our main theorem. Let $m$ be an integer.

\begin{lemma} \label{lem:main2}
Let $a_1, \ldots, a_r$ be rational numbers. Suppose that for each $1 \le i \le r$ and each $n \in \bN$ we have an expression
\begin{displaymath}
\binom{a_i}{n} = P_{i,n}(a_1, \ldots, a_r)
\end{displaymath}
where $P_{i,n} \in \bZ[1/m][T_1, \ldots, T_r]$ is a polynomial of total degree $\le n$. Then each $a_i$ belongs to $\bZ[1/m]$.
\end{lemma}

\begin{proof}
Let $p \nmid m$ be a prime number, and write $v_p$ for the $p$-adic valuation on rational numbers. Suppose $v_p(a_i)<0$ for some $i$. Relabeling if necessary, assume $v_p(a_1)=-v$ and $v_p(a_i) \ge -v$ for all $i$. Now, if $\ell$ is an integer then $v_p(a_1-\ell)=-v$. We thus see that $v_p(\binom{a_1}{n})=-nv-v_p(n!)$; in particular, $v_p(\binom{a_1}{p})=-nv-1$. If $b$ is any degree $n$ polnomial expression in the $a_i$'s with coefficients in $\bZ[1/m]$ then $v_p(b) \ge -nv$. We thus cannot have $\binom{a_1}{p}=b$, and so we have obtained a contradiction. We conclude that $v_p(a_i) \ge 0$ for each $i$. Since this holds for all $p$ not dividing $m$, each $a_i$ belongs to $\bZ[1/m]$.
\end{proof}

\begin{lemma} \label{lem:main3}
Let $a \in \bZ[1/m]$. Suppose $a$ divides $\binom{a}{n}$ in $\bZ[1/m]$ for each $n \ge 1$. Then $a$ is either~0 or a unit of $\bZ[1/m]$.
\end{lemma}

\begin{proof}
Suppose $a \ne 0$ and $p$ is a prime number that divides $a$ and does not divide $m$. Let $v$ be the $p$-adic valuation of $a$. The numbers $a-1, \ldots, a-p+1$ are coprime to $p$, as are the numbers $1, \ldots, p-1$, and so the $p$-adic valuation of $\binom{a}{p}$ is $v-1$. Thus $a$ does not divide $\binom{a}{p}$ in $\bZ[1/m]$, a contradiction.
\end{proof}

\subsection{The main theorem} \label{ss:main}

The following is our main theorem.

\begin{theorem} \label{mainthm1}
Suppose $\IND(m)$ and $\GEN(m)$ hold. Then:
\begin{enumerate}
\item We have a ring isomorphism $\Theta(G, \Omega)[1/m] = \bZ[1/m]^n$ for some $n$.
\item Let $f \colon Y \to X$ be a map of transitive $\sE(\Omega)$-smooth $G$-sets, and let $(a_1, \ldots, a_n)$ be the corresponding element of $\bZ[1/m]^n$. Then each $a_i$ is either~0 or a unit of $\bZ[1/m]$.
\item If $\Theta_{\le 1}(G, \Omega)[1/m]$ is generated by $r$ elements as a $\bZ[1/m]$-module and $p$ is the smallest prime not dividing $m$ then $n \le p^r$.
\end{enumerate}
\end{theorem}

\begin{proof}
The assumptions imply that $\Theta(G, \Omega)[1/m]$ is a finitely generated binomial ring (Theorem~\ref{thm:theta-binom} and Corollary~\ref{cor:theta-fg}). The structure theorem for such rings \cite[Theorem~9]{Xantcha} thus gives a ring isomorphism
\begin{displaymath}
\Theta(G, \Omega)[1/m] = \bZ[1/m_1] \times \cdots \times \bZ[1/m_n]
\end{displaymath}
for positive integers $m_1, \ldots, m_n$, each divisible by $m$.

(a) Consider a ring homomorphism $\mu \colon \Theta(G, \Omega)[1/m] \to \bQ$. Let $x_1, \ldots, x_r$ be basic classes that generate $\Theta_{\le 1}(G, \Omega)[1/m]$ as a $\bZ[1/m]$-module and let $a_i=\phi(x_i)$. By Proposition~\ref{prop:pi}(a), $\binom{a_i}{n}$ is a degree $\le n$ polynomial in $a_1, \ldots, a_r$ with coefficients in $\bZ[1/m]$. Lemma~\ref{lem:main2} thus shows that each $a_i$ belongs to $\bZ[1/m]$. Hence the image of $\mu$ is just $\bZ[1/m]$. This shows that the $m_i$'s in the previous paragraph are all equal to $m$, which proves (a).

(b) One again, consider a ring homomorphism $\mu \colon \Theta(G, \Omega)[1/m] \to \bZ[1/m]$. Let $x$ be a basic class and let $a=\mu(x)$. By Proposition~\ref{prop:pi}(b), $a$ divides $\binom{a}{n}$ in $\bZ[1/m]$ for all $n \ge 1$. Thus $a$ is either~0 or a unit of $\bZ[1/m]$ by Lemma~\ref{lem:main3}. If $f \colon Y \to X$ is a map of transitive $\sE(\Omega)$-smooth sets then $[f]$ is a product of basic classes by Proposition~\ref{prop:filt}(b), and so $\mu(f)$ is also~0 or a unit. The result follows.

(c) Let $x_1, \dots, x_r$ generate $\Theta_{\le 1}(G, \Omega)[1/m]$ as a $\bZ[1/m]$-module. Then the $x_i$'s generate $\Theta(G, \Omega)[1/m]$ as a $\bZ[1/m]$-algebra by Proposition~\ref{prop:filt}. Reducing modulo $p$, the images of $x_1, \ldots, x_r$ generate the ring
\begin{displaymath}
\Theta(G, \Omega)[1/m] \otimes \bF_p \cong \bF_p^n
\end{displaymath}
Since any element $a$ of the ring $\bF_p^n$ satisfies Fermat's little theorem ($a^p=a$), it follows that we have a surjection
\begin{displaymath}
\bF_p[T_1, \ldots, T_r]/(T_i^p-T_i) \to \bF_p^n, \qquad T_i \mapsto x_i.
\end{displaymath}
The left side has dimension $p^r$ as an $\bF_p$-vector space, and so the result follows.
\end{proof}

\begin{remark} \label{rmk:refined}
Suppose $m=m_1 m_2$ where $m_1$ and $m_2$ are coprime. Consider the following refined version of $\IND(m)$:
\begin{description}[align=right,labelwidth=2.75cm,leftmargin=!,font=\normalfont]
\item[$\IND(m_1, m_2)$] For finite $A \subset \Omega$ the index of $G(A)$ in $G[A]$ divides $m_1 m_2^s$ for some $s \ge 0$.
\end{description}
In other words, the primes dividing $m_1$ can appear in the above indices, but they have bounded exponent. Using this condition, we can obtain a refinement of our theorem: assuming $\GEN(m_2)$ and $\IND(m_1, m_2)$, the $a_i$'s in Theorem~\ref{mainthm1}(b) belong to $\bZ[1/m_2]$, i.e., the primes dividing $m_1$ cannot appear in the denominators. This situation actually occurs with certain classes of trees; see \S \ref{ss:ex-details}(c).

We briefly sketch the argument. First, in Proposition~\ref{prop:pi}(a), we find that $\binom{x}{n}$ belongs to $m_1^{-1} \Theta_{\le n}(G, \Omega)[1/m_2]$; this follows since the $d_i$'s appearing there divide $m_1 m_2^s$ for some $s$. Next, we can adapt Lemma~\ref{lem:main2}: if we allow the $P$'s to have cofficients in $m_1^{-1} \bZ[1/m_2]$ then the proof shows that the $x_i$'s belong to $\bZ[1/m_2]$. The proof of Theorem~\ref{mainthm1}(a) now yields the claim.
\end{remark}

We restate the theorem in the $m=1$ case, as the statements simplify somewhat. Note that $\IND(1)$ means that $G(A)=G[A]$ for all finite $A \subset \Omega$.

\begin{corollary} \label{cor:m1}
Suppose $\GEN(1)$ and $\IND(1)$ hold. Then:
\begin{enumerate}
\item We have a ring isomorphism $\Theta(G, \Omega) = \bZ^n$ for some $n$.
\item If $f \colon Y \to X$ is a map of transitive $\sE(\Omega)$-smooth $G$-sets, and $(a_1, \ldots, a_n)$ is the corresponding element of $\bZ^n$ then each $a_i$ belongs to $\{-1, 0, 1\}$.
\item If $\Theta_{\le 1}(G, \Omega)$ has rank $r$ then $n \le 2^r$.
\end{enumerate}
\end{corollary}

The above result applies to the cases analyzed in \cite{colored} and \cite{homoperm}, and provides a conceptual explanation for why the measures we found there take values in $\{-1,0,1\}$.

\subsection{Regular measures} \label{ss:reg}

A $k$-valued measure $\mu$ for $(G, \Omega)$ is \defn{regular} if $\mu(f)$ is a unit of $k$ whenever $f$ is a map of transitive $\sE(\Omega)$-smooth $G$-sets. Regular measures play an important role in the application to tensor categories, as they often lead to semi-simple categories. There is a universal ring $\Theta^*(G, \Omega)$ for regular measures, obtained by localizing $\Theta(G, \Omega)$ at the classes $[f]$ of the said $f$'s. The following is our main result about this ring.

\begin{theorem}
Suppose $\IND(m)$ and $\GEN(m)$ hold. Then:
\begin{enumerate}
\item We have a ring isomorphism $\Theta^*(G, \Omega)[1/m] = \bZ[1/m]^{\ell}$ for some ${\ell}$.
\item Let $f \colon Y \to X$ be a map of transitive $\sE(\Omega)$-smooth $G$-sets, and let $(a_1, \ldots, a_{\ell})$ be the corresponding element of $\bZ[1/m]^{\ell}$. Then each $a_i$ is a unit of $\bZ[1/m]$.
\item If $\Theta_{\le 1}(G, \Omega)[1/m]$ is generated by $r$ elements as a $\bZ[1/m]$-module and $p$ is the smallest prime not dividing $m$ then $\ell \le (p-1)^r$.
\end{enumerate}
\end{theorem}

\begin{proof}
Statements (a) and (b) follow from Theorem~\ref{mainthm1} and the description of $\Theta^*(G, \Omega)$ as a localization of $\Theta(G, \Omega)$. For (c), suppose $x_1, \ldots, x_r$ generate $\Theta_{\le 1}(G, \Omega)[1/m]$. Then the $x_i$'s generate $\Theta(G, \Omega)[1/m]$ as a $\bZ[1/m]$-algebra, and also generate $\Theta^*(G, \Omega)[1/m]$ as a $\bZ[1/m]$-algebra as well, as the latter is a quotient of the former in this case. Reducing modulo $p$, the images of the $x_i$'s generate $\bF_p^{\ell}$. The image of each $x_i$ satisfies the equation $a^{p-1}=1$, and so we have a surjection
\begin{displaymath}
\bF_p[T_1, \ldots, T_r]/(T_i^{p-1}-1) \to \bF^{\ell}
\end{displaymath}
As the domain has dimension $(p-1)^r$, the result follows.
\end{proof}

The above result is particularly interesting when $m$ is odd, for then we obtain the bound $\ell \le 1$, i.e., a regular measure is unique if it exists. We now analyze this situation more closely. We require the following concept: we say that $(G, \Omega)$ is \defn{odd} if whenever $Y \to X$ and $Z \to X$ are maps of transitive $\sE(\Omega)$-smooth $G$-sets, the fiber product $Y \times_X Z$ has an odd number of $G$-orbits. The following simple observation, which holds in complete generality, shows why this is a relevant condition (see also \cite[Proposition~8.7]{regcat}).

\begin{proposition} \label{prop:odd}
An oligomorphic group $(G, \Omega)$ admits a regular $\bF_2$-valued measure if and only if it is odd, in which case the measure is unique.
\end{proposition}

\begin{proof}
If $\mu$ is a regular $\bF_2$-valued measure then we must necessarily have $\mu(f)=1$ whenever $f$ is a map of transitive $\sE(\Omega)$-smooth $G$-sets. If we define $\mu$ in this manner and extend it additively then conditions (a), (b), and (c) of Definition~\ref{defn:meas}(a,b,c) hold automatically, while (d) is equivalent to the oddness condition.
\end{proof}

We can now completely determine $\Theta^*[1/m]$ when $m$ is odd.

\begin{theorem} \label{thm:reg}
Suppose $\GEN(m)$ and $\IND(m)$ hold with $m$ odd. Then
\begin{displaymath}
\Theta^*(G, \Omega)[1/m] = \begin{cases}
\bZ[1/m] & \text{if $(G, \Omega)$ is odd} \\
0 & \text{otherwise.} \end{cases}
\end{displaymath}
\end{theorem}

\begin{proof}
We have seen that $\Theta^*(G, \Omega)[1/m]$ is isomorphic to $\bZ[1/m]^{\ell}$ where $\ell$ is either~0 or~1. We have $\ell=1$ if and only if there is a homomorphism $\Theta^*(G, \Omega)[1/m] \to \bF_2$. By Proposition~\ref{prop:odd}, such a homomorphism exists if and only if $(G, \Omega)$ is odd.
\end{proof}

\begin{remark}
Even though there is always at most one regular $\bF_2$-valued measure, the ring $\Theta^*(G, \Omega)$ can be quite complicated in general. For example, in the case of the infinite symmetric group it is the localization of $\bZ[t]$ obtained by inverting the elements $t-n$ for all $n \in \bN$. Thus the conclusion in Theorem~\ref{thm:reg} is quite strong.
\end{remark}

\begin{remark}
Assuming $\GEN(m)$ and $\IND(m)$ with $m$ odd, Theorem~\ref{thm:reg} implies that there is at most one semi-simple $\bC$-linear tensor category associated to $G$ via the theory of \cite{repst}.
\end{remark}

Suppose now that $\GEN(1)$ and $\IND(1)$ hold and $(G, \Omega)$ is odd. We thus have a unique $\bZ$-valued regular measure $\mu$. If $f \colon Y \to X$ is a map of transitive $\sE(\Omega)$-smooth $G$-sets then $\mu(f)=\pm 1$ by Corollary~\ref{cor:m1}. The following proposition computes the sign, and thus completely determines the measure $\mu$.

\begin{proposition}
Maintain the above notation, and let $n$ be number of $G$-orbits on $Y^{(2)}_{/X}$. Then $\mu(f) = (-1)^n$.
\end{proposition}

\begin{proof}
Let $x=[f] \in \Theta(G, \Omega)$, let $g \colon Y^{(2)}_{/X} \to X$ be the natural map, and let $y=[g]$. We have $\binom{x}{2}=y$ in $\Theta(G, \Omega)$ by Proposition~\ref{prop:bcoeff} and the reasoning in the proof of Proposition~\ref{prop:pi}. Now apply $\mu$ and reduce modulo~2. Letting $a=\mu(f)$, we find $\binom{a}{2}=n$ modulo~2, and so the result follows.
\end{proof}

\section{Translation to Fra\"iss\'e classes} \label{s:struct}

\subsection{Overview}

Oligomorphic groups and Fra\"iss\'e classes are, in many respects, two sides of the same coin. We have worked with oligomorphic groups up to now, since our proofs are easier to write in that language. We now explain our results in the Fra\"iss\'e class language. This perspective is crucial for describing examples and carrying out concrete calculations.

There is also one new ingredient we add to the mix in this section, namely, the notion of a minimal marked structure. The classes of these structures generate $\Theta_{\le 1}(\fF)$ as an additive group. In particular, if there are finitely many minimal marked structures (the ``distal'' condition) then $\GEN(1)$ holds. This gives a purely combinatorial mechanism for verifying this condition.

\subsection{Fra\"iss\'e classes}

We begin by recalling some background on Fra\"iss\'e classes. We refer to \cite[\S 6.2]{repst}, \cite{CameronBook}, or \cite{Macpherson} for additional background.

Let $\fF$ be a class of finite relational structures for some fixed signature. Thus each member of $\fF$ is a finite set $X$ equipped with a family $\{R_i\}_{i \in I}$ of relations of varying arities; the index set $I$ and the arities are specified by the signature. Suppose that $i \colon X \to Y$ and $j \colon X \to X'$ are embeddings of structures in $\fF$. An \defn{amalgamation} of $Y$ and $X'$ over $X$ is a structure $Y'$ in $\fF$ equipped with embeddings $i' \colon X' \to Y'$ and $j' \colon Y \to Y'$ such that $i'j=j'i$ and $Y'=\operatorname{im}(i') \cup \operatorname{im}(j')$. We say that $\fF$ has the \defn{amalgamation property} if at least one amalgmation exists (for any given $i$ and $j$). This is the key condition in the notion of Fra\"iss\'e class:

\begin{definition} \label{defn:fraisse}
The class $\fF$ is a \defn{Fra\"iss\'e class} if the following conditions hold:
\begin{enumerate}
\item $\fF$ has only finitely many isomorphism classes of $n$-element structures\footnote{This is often weakened to simply asking that $\fF$ contains countably many isomorphism classes.}, for each $n$.
\item $\fF$ is hereditary: if $X \in \fF$ and $Y$ embeds into $X$ then $Y \in \fF$.
\item $\fF$ has the amalgamation property. \qedhere
\end{enumerate}
\end{definition}

Assume that $\fF$ is a Fra\"iss\'e class. Then $\fF$ admits a \defn{Fra\"iss\'e limit}. This is a countable homogeneous structure $\Omega$ whose age is equal to $\fF$, and it is unique up to isomorphism. \textit{Homogeneous} means that if $i,j \colon X \to \Omega$ are two embeddings of a finite structure $X$ then there is an automorphism $\sigma$ of $\Omega$ such that $j=\sigma \circ i$. The \defn{age} of $\Omega$ is simply the class of all finite structures that embed into it. Let $G$ be the automorphism group of $\Omega$. It follows easily from homogeneity that the action of $G$ on $\Omega$ is oligomorphic.

For a structure $X \in \fF$, we let $\Omega^{[X]}$ denote the set of embeddings $X \to \Omega$. The homogeneity of $\Omega$ implies that this is a transitive $G$-set. One easily sees that each $\sE(\Omega)$-smooth transitive $G$-set is isomorphic to some $\Omega^{[X]}$. We warn the reader, however, that it is possible for $\Omega^{[X]}$ and $\Omega^{[Y]}$ to be isomorphic $G$-sets even if $X$ and $Y$ are not isomorphic structures\footnote{This happens when $X$ and $Y$ have isomorphic definable closures.}. If $i \colon X \to Y$ is an embedding of finite structures then there is an induced map $i^* \colon \Omega^{[Y]} \to \Omega^{[X]}$ of $G$-sets.

We fix $\fF$, $G$, and $\Omega$ as above for the remainder of \S \ref{s:struct}.

\subsection{Automorphisms} \label{ss:aut}

Consider the following condition on $\fF$:
\begin{description}[align=right,labelwidth=1.75cm,leftmargin=!,font=\normalfont]
\item[$\AUT(m)$] If $X$ is a member of $\fF$ then the order of $\Aut(X)$ divides a power of $m$.
\end{description}
This is exactly the ``bounded automorphism group'' condition appearing in Theorem~\ref{mainthm}(b). We now show that this matches the $\IND(m)$ condition previously considered for $G$.

\begin{proposition} \label{prop:aut}
$\fF$ satisfies $\AUT(m)$ if and only if $(G, \Omega)$ satisfies $\IND(m)$.
\end{proposition}

\begin{proof}
Let $A$ be a finite subset of $\Omega$; we give $A$ the induced relational structure. Every member of $\fF$ is obtained in this way (up to isomorphism). The map $G[A]/G(A) \to \Aut(A)$ is an isomorphism; indeed, it is clearly injective, and the homogeneity of $\Omega$ (giving a map in the opposite direction) implies that it is surjective. The result thus follows.
\end{proof}

\subsection{Separated sets} \label{ss:sep}

We now discuss a notion of separatedness, which figures into the notion of a minimal marked structure. Let $X$ be a member of $\fF$, and let $A$ and $B$ be disjoint subsets of $X$. We say that $A$ and $B$ are \defn{separated} if $X$ is the unique amalgamation of $X \setminus A$ and $X \setminus B$ over $X \setminus (A \cup B)$. Equivalently, $A$ and $B$ are separated if and only if the square
\begin{displaymath}
\xymatrix{
\Omega^{[X]} \ar[r] \ar[d] & \Omega^{[X \setminus A]} \ar[d] \\
\Omega^{[X \setminus B]} \ar[r] & \Omega^{[X \setminus (A \cup B)]} }
\end{displaymath}
is cartesian; indeed, if there an amalgamation other than $X$, then the fiber product would have more than one orbit.

We say that two vertices $x,y \in X$ are \defn{separated} if the corresponding singleton sets are. We now discuss this condition in a little more detail. We say that $x$ and $y$ are \defn{similar} if there is an automorphism of $X$ switching $x$ and $y$ and fixing all other elements. In this case, $x$ and $y$ are not separated: there is an amalgamation of $X \setminus x$ and $X \setminus y$ over $X \setminus \{x,y\}$ where $x$ and $y$ are identified. Suppose now that $x$ and $y$ are dissimilar, and let $X'$ be an amalgamation of $X \setminus x$ and $X \setminus y$ over $X \setminus \{y,x\}$. We can then identify $X'$ with $X$ as a set. If $\{R_i\}$ and $\{R'_i\}$ are the relations on $X$ and $X'$ then we have $R_i(z_1, \ldots, z_{n(i)})=R'_i(z_1, \ldots, z_{n(i)})$ for all $i$ and $z_1, \ldots, z_{n(i)} \in X$, provided that $x$ and $y$ do not both belong to $\{z_1, \ldots, z_{n(i)}\}$. Thus separatedness means that one can conclude $R_i=R'_i$ from this condition.

\begin{example}
We give two examples. (a) If $\fF$ is the class of finite totally ordered sets then $x$ and $y$ are separated if and only if they are not adjacent, i.e., there exists some $a \in X$ such that $x<a<y$ or $y<a<x$ \cite[Example~2.11]{arboreal}. (b) If $\fF$ is the class of finite sets then no pair of elements is separated \cite[Example~2.10]{arboreal} for more details.
\end{example}

We will require the following property of separatedness:

\begin{proposition} \label{prop:sep}
Let $X$ be a structure in $\fF$ and let $A$, $B$, and $C$ be disjoint subsets of $X$. Then the following are equivalent:
\begin{enumerate}
\item $A \cup B$ is separated from $C$ in $X$
\item $A$ is separated from $C$ in $X$, and $B$ is separated from $C$ in $X \setminus A$.
\end{enumerate}
\end{proposition}

\begin{proof}
Consider the square
\begin{displaymath}
\xymatrix{
E \ar[r] \ar[d] & D \ar[r] \ar[d] & \Omega^{[X \setminus (A \cup B)]} \ar[d] \\
\Omega^{[X \setminus C]} \ar[r] & \Omega^{[X \setminus (A \cup C)]} \ar[r] & \Omega^{[X \setminus (A \cup B \cup C)]}
}
\end{displaymath}
where the two squares are cartesian (and so the outer rectangle is cartesian as well). If (a) holds then $E=\Omega^{[X]}$ is transitive, and so $D$ must be transitive as well, and thus equal to $\Omega^{[X \setminus A]}$; thus (b) holds. Conversely, if (b) holds, then clearly $E=\Omega^{[X]}$, and so (a) holds.
\end{proof}

\subsection{Minimal marked structures} \label{ss:marked}

A \defn{marked structure} is a pair $(X,x)$ where $X$ is a member of $\fF$ and $x$ is a point of $X$. We say that an element $y \ne x$ of a marked structure $(X,x)$ is \defn{extraneous} if it is separated from the marked point $x$. We say that a marked structure is \defn{minimal} if there are no extraneous points. We will be interested in the following condition:
\begin{description}[align=right,labelwidth=1.5cm,leftmargin=!,font=\normalfont]
\item[$\FMM$] $\fF$ has finitely many minimal marked structures (up to isomorphism).
\end{description}
We give a few examples:

\begin{example} \label{ex:fmm}
(a) If $\fF$ is the class of finite totally ordered sets then there are four minimal marked structures, namely, the three marked structures with $\le 2$ points, and the three element set where the middle element is marked, and so $\FMM$ holds \cite[Example~2.11]{arboreal}. (b) If $\fF$ is the class of finite trees with valence $\le n$ then $\FMM$ holds \cite[\S 3.4]{arboreal}. (c) If $\fF$ is the class of finite sets then every marked structure is minimal, so $\FMM$ does not hold.
\end{example}

We will require the following consequence of $\FMM$.

\begin{proposition} \label{prop:extra}
Suppose $\FMM$ holds, and let $r$ be the maximal size of a minimal marked structure. If $(X,x)$ is a marked structure with $n$ elements then at least $n-r$ elements of $X$ are extraneous.
\end{proposition}

\begin{proof}
By definition, there is a sequence of distinct elements $y_1, \ldots, y_s$ in $X \setminus x$, with $s \ge n-r$, such that $y_i$ is extraneous in $X \setminus \{y_1, \ldots, y_{i-1}\}$ for all $1 \le i \le s$. By Proposition~\ref{prop:sep}, the set $\{y_1, \ldots, y_s\}$ is separated from $x$ in $X$. That proposition again implies that each $y_i$ is separated from $x$ in $X$, which completes the proof.
\end{proof}

Finally, we observe that $\FMM$ is equivalent to the notion of \defn{distality} introduced by Simon \cite{Simon}. We note that nothing in this paper relies on this equivalence, and we prove it simply to connect our condition with something already in the literature.

\begin{proposition} \label{prop:FMM_distal}
$\fF$ has $\FMM$ if and only if $\Omega$ is distal.
\end{proposition}

\begin{proof}
We recall that distality for a homogeneous structure $\Omega$ is equivalent to the following condition \cite[Definition 2.34]{OnSim}:
\begin{description}[align=right,labelwidth=1.5cm,leftmargin=!,font=\normalfont]
\item[$\DIS$] there is an integer $k$ such that for any finite set $A \subset \Omega$ and singleton $a\in \Omega$, there is $A_0 \subseteq A$ of size $\leqslant k$ such that $\type(a/A_0) \vdash \type(a/A)$, meaning that whenever $\type(a'/A_0) = \type(a/A_0)$ we have $\type(a'/A) = \type(a/A)$.
\end{description}
Due to oligomorphicity, we can instead simply ask that there exists a finite collection $\{A_0\}$ of finite structures as above.
	
The type $\type(x/X)$ corresponds to the $G(X)$-orbit of $x$ inside $\Omega$; this orbit is indexed by a marked structure $(X \cup x,x)$, where $X\cup x$ denotes the structure induced on a subset $X \cup x \subset \Omega$. Therefore the conclusion ``$\type(a/A_0) \vdash \type(a/A)$'' of $\DIS$ can be reformulated as ``the projection $p \colon \Omega^{[A]} \rightarrow \Omega^{[A_0]}$ induces a bijection between the orbits $G(A) \cdot a$ and $G(A_0)\cdot a$''. This condition is equivalent to the uniqueness of the amalgamation:
\begin{displaymath}
		\xymatrix{
			A \cup a & A \ar[l] \\
			A_0 \cup a \ar[u] & A_0 \ar[l] \ar[u] }
\end{displaymath}
Finally, we see that $\DIS$ is equivalent to 
\begin{description}[align=right,labelwidth=1.5cm,leftmargin=!,font=\normalfont]
\item[$\DIS'$] there are finitely many structures $\{A_0\}$ such that for any one-point extension $A \rightarrow A\cup a$, there is $A_0$ such that the amalgamation of $A_0 \cup a$ and $A$ over $A_0$ is unique.
\end{description}
As $\DIS'$ is equivalent to $\FMM$, the result follows.
\end{proof}

\subsection{Measures} \label{ss:Fmeas}

The next definition comes from \cite[\S 6.2]{repst}.

\begin{definition} \label{defn:Fmeas}
A \defn{measure} on $\fF$ valued in a commutative ring $k$ is a rule $\mu$ assigning to each embedding $i \colon X \to Y$ of structures in $\fF$ a quantity $\mu(i)$ in $k$ such that the following conditions hold:
\begin{enumerate}
\item $\mu(i)=1$ if $i$ is an isomorphism.
\item If $i \colon X \to Y$ and $j \colon Y \to Z$ are composable embeddings then $\mu(j \circ i)=\mu(j) \cdot \mu(i)$.
\item Let $i \colon X \to Y$ and $j \colon X \to X'$ be embeddings, and let $(Y'_{\alpha}, i'_{\alpha}, j'_{\alpha})$ for $1 \le \alpha \le n$ be the various amalgamations. Then $\mu(i)=\sum_{\alpha=1}^n \mu(i'_{\alpha})$. \qedhere
\end{enumerate}
\end{definition}

As in the group case, there is a universal measure valued in a ring $\Theta(\fF)$. Each embedding $i \colon X \to Y$ as above defines a class $[i]$ in $\Theta(\fF)$, and these classes satisfy relations analogous to (a)--(c) above. In \cite[\S 6.9]{repst}, we show that measures for $G$ and measures for $\fF$ correspond. The following is the precise statement:

\begin{proposition} \label{prop:Theta-isom}
We have a natural ring isomorphism $\Theta(\fF) = \Theta(G, \Omega)$. Under this isomorphism, the class of an embedding $i \colon X \to Y$ corresponds to the class of the $G$-map $i^* \colon \Omega^{[Y]} \to \Omega^{[X]}$.
\end{proposition}

We can now translate our binomiality result into purely combinatorial terms:

\begin{corollary}
If $\AUT(m)$ holds then $\Theta(\fF)[1/m]$ is a binomial ring.
\end{corollary}

\begin{proof}
This follows from Theorem~\ref{thm:theta-binom} and Proposition~\ref{prop:aut}.
\end{proof}

\subsection{The filtration on $\Theta(\fF)$}

We say that an embedding $i \colon X \to Y$ is a \defn{one-point extension} if $\# Y=1+\# X$. Every embedding factors into a sequence of one-point extensions, and so the classes of one-point extensions generate $\Theta(\fF)$ by Definition~\ref{defn:Fmeas}(b). A marked structure $(X,x)$ defines a one-point extension $X \setminus x \to X$, and every one-point extension is isomorphic to one of this form. We let $[X,x]$ be the class in $\Theta(\fF)$ of the associated one-point extension. Thus the classes of marked structures generate $\Theta(\fF)$ as a ring.

We define $\Theta_{\le 1}(\fF)$ be the $\bZ$-submodule of $\Theta(\fF)$ generated by the classes of marked structures. We say that $\fF$ satisfies $\GEN(m)$ if $\Theta_{\le 1}(\fF)[1/m]$ is finite generated as a $\bZ[1/m]$-module. The isomorphism from Proposition~\ref{prop:Theta-isom} is compatible with filtrations, i.e., $\Theta_{\le 1}(\fF)$ maps isomorphically to $\Theta_{\le 1}(G, \Omega)$, and so $\GEN(m)$ holds for $\fF$ if and only if it holds for $(G, \Omega)$. The following result is the main reason we care about minimal marked structures:

\begin{proposition} \label{prop:min-gen}
The minimal marked structures generated $\Theta_{\le 1}(\fF)$ as a $\bZ$-module. In particular, if $\FMM$ holds then $\fF$ satisfies $\GEN(1)$.
\end{proposition}

\begin{proof}
If $y$ is extraneous in the markested structure $(X,x)$ then $[X,x]=[X \setminus y, x]$ holds in $\Theta(\fF)$. This follows immediately from the definition of extraneous and axiom Definition~\ref{defn:Fmeas}(a); see \cite[Proposition~2.9]{arboreal} for details. The result thus follows.
\end{proof}

\begin{example} \label{ex:FMM-vs-Fin}
Let $\fF$ be the class of finite sets. Then $\Theta(\fF)=\bZ[t]$. If $X$ is an $n$-element marked set then $[X]$ is identified with $t-n+1$ under this isomorphism. In particular, $\Theta_{\le 1}(\fF)$ consists of those elements of the form $a+bt$, and is therefore a finitely generated $\bZ$-module. We therefore see that $\fF$ satisfies $\GEN(1)$ but does not satisfy $\FMM$.
\end{example}

\begin{remark}
It follows from Proposition~\ref{prop:min-gen} that the classes of minimal marked structures generated $\Theta(\fF)$ as a ring. In \cite{thesis} a presentation for $\Theta(\fF)$ is given in terms of these generators; see also \cite[\S 2.6]{arboreal}.
\end{remark}

\subsection{The main theorem}

We now reformulate Theorem~\ref{mainthm1} in terms of $\fF$.

\begin{theorem} \label{mainthm2}
Suppose $\fF$ satisfies $\AUT(m)$ and $\GEN(m)$. Then the conclusions of Theorem~\ref{mainthm1} apply to $\Theta(\fF)$.
\end{theorem}

\begin{proof}
We have seen that the conditions $\AUT(m)$ and $\GEN(m)$ on $\fF$ correspond to the conditions $\IND(m)$ and $\GEN(m)$ on $(G, \Omega)$. Thus the result follows from Theorem~\ref{mainthm1} and the isomorphism $\Theta(\fF)=\Theta(G, \Omega)$.
\end{proof}

Since $\FMM$ implies $\GEN(1)$, the conclusions of the theorem apply if we assume $\fF$ satisfies $\AUT(m)$ and $\FMM$. This yields Theorem~\ref{mainthm} from the introduction. One can translate the results from \S \ref{ss:reg} on regular measures to the present setting as well. In particular, Theorem~\ref{thm:reg} yields Theorem~\ref{regthm}.

\subsection{Joins}

We now discuss a general construction that can be used to produce many Fra\"iss\'e classes satisfying the conditions of Theorem~\ref{mainthm2}. Suppose $\fF$ and $\fF'$ are two classes of finite relational structures. We define $\fF \ast \fF'$ to be the class of all finite sets simultaneously equipped with structure from $\fF$ and structure from $\fF'$. For example, if $\fF$ is the class of finite graphs and $\fF'$ is the class of finite totally ordered sets then $\fF \ast \fF'$ is the class of finite graphs with a total order on the vertex set; there is no connection between the graph structure and order. If $\fF$ and $\fF'$ are both Fra\"iss\'e classes satisfying the strong amalgamation property then $\fF \ast \fF'$ is also a Fra\"iss\'e class satisfying strong amalgamation \cite[\S 3.9]{CameronBook}.

\begin{proposition} \label{prop:star}
Suppose $\fF$ and $\fF'$ are Fra\"iss\'e classes with strong amalgamation.
\begin{enumerate}
\item If either $\fF$ or $\fF'$ satisfies $\AUT(m)$ then so does $\fF \star \fF'$.
\item If $\fF$ and $\fF'$ both satisfy $\FMM$ then so does $\fF \star \fF'$.
\end{enumerate}
\end{proposition}

\begin{proof}
(a) Let $X$ be a member of $\fF \ast \fF'$. A permutation of the set $X$ preserves the $\fF \star \fF'$ structure if and only if it preserves both the $\fF$ structure and the $\fF'$ structure. We thus have
\begin{displaymath}
\Aut_{\fF \star \fF'}(X) = \Aut_{\fF}(X) \cap \Aut_{\fF'}(X),
\end{displaymath}
considering all three as subgroups of $\mathrm{Sym}(X)$, from which the statement follows.

(b) Suppose that every minimal marked structure of $\fF$ has $\le r$ elements, and define $r'$ analogously for $\fF'$. Suppose $(X, x)$ is a marked structure in $\fF \star \fF'$ with $\ge r+r'$ vertices. By Proposition~\ref{prop:extra}, there is an element $y \ne x$ that is extraneous in both the $\fF$ and $\fF'$ structures separately. But this implies that $y$ is extraneous in the $\fF \star \fF'$ structure; this is easy to see using the $R_i=R'_i$ perspective on separatedness discussed in \S \ref{ss:sep}. Thus any minimal marked structure in $\fF \star \fF'$ has $\le r+r'$ elements, and so $\FMM$ holds.
\end{proof}

Let $\fC_s$ be the class of $s$-colored finite sets. This is a Fra\"iss\'e class and clearly satisfies strong amalgamation. We let $\fF_s=\fC_s \ast \fF$, which we regard as the class of $s$-colored structures in $\fF$. If $\fF$ is a Fra\"iss\'e class with strong amalgamation then $\fF_s$ is also a Fra\"iss\'e class with strong amalgamation; moreover, if $\fF$ satisfies $\AUT(m)$ then so does $\fF_s$. We now show that $\fF_s$ satisfies $\FMM$ when $\fF$ does; note that the above proposition does not give this since $\fC_s$ does not satisfy $\FMM$ (see Example~\ref{ex:fmm}).

\begin{proposition} \label{prop:colored}
Given a minimimal marked structure in $\fF_s$, the underlying marked structure in $\fF$ is also minimal. In particular, if $\fF$ satisfies $\FMM$ then so does $\fF_s$.
\end{proposition}

\begin{proof}
Let $X$ be a structure in $\fF_s$. Suppose that $x,y \in X$ are $\fF$-separated. We claim that they are $\fF_s$-separated. Clearly, since $x$ and $y$ are not $\fF$-similar they are not $\fF_s$-similar. The claim now follows easily from the $R_i=R'_i$ perspective from \S \ref{ss:sep}. If $(X,x)$ is now a marked structure in $\fF_s$ then any point that is $\fF$-extraneous is also $\fF_s$-extraneous; thus if $(X,x)$ is $\fF_s$-minimal then it is $\fF$-minimal.
\end{proof}

\subsection{Examples} \label{ss:ex-details}

We now provide some details on the examples from \S \ref{ss:ex}.

(a) Let $\fF_s$ be the class of finite sets equipped with $s$ total orders. Thus $\fF_s=\fF_1^{\star s}$. Since $\fF_1$ satisfies strong amalgamation, it follows that $\fF_s$ is a Fra\"iss\'e class. Since $\fF_1$ also satisfies $\FMM$ (Example~\ref{ex:fmm}) and $\AUT(1)$ (obvious), so does $\fF_s$ by Proposition~\ref{prop:star}. Thus Theorem~\ref{mainthm2} applies to $\fF_s$ with $m=1$. In \S \ref{ss:ex}, we claimed that a minimal marked structure in $\fF_s$ has at most $2s+1$ points. This can be proved using arguments as in \cite[\S 3.2]{homoperm}.

(b) Let $\fF_s$ be the class of finite sets equipped with a total order and an $s$-coloring. This is a Fra\"iss\'e class and obviously satisfies $\AUT(1)$. By Example~\ref{ex:fmm} and Proposition~\ref{prop:colored}, any minimal marked structure in $\fF_s$ has at most three points, and so, in particular, $\FMM$ holds. Thus Theorem~\ref{mainthm2} applies to $\fF_s$ with $m=1$.

(c) Let $\fF$ be the class of boron trees, and let $\fF_s=\fC_s \star \fF$ be the $s$-colored version. Since $\fF$ is a Fra\"iss\'e class satisfying strong amalgamation (see the ``First variation'' following \cite[Proposition~3.2]{CameronTrees}), it follows that $\fF_s$ is a Fra\"iss\'e class. Any automorphism group in $\fF$ has order $3 \cdot 2^r$ for some $r$ \cite[Lemma~17.13]{repst}, and so the same is true for $\fF_s$. In particular, $\fF_s$ satisfies $\AUT(6)$, and even the refined version $\AUT(3,2)$ from Remark~\ref{rmk:refined}. The analysis from \cite[\S 3.4]{arboreal} shows that the minimal marked structures in $\fF$ have $\le 5$ elements, and so the same is true for $\fF_s$ by Proposition~\ref{prop:colored}; in particular, $\FMM$ holds. Thus Theorem~\ref{mainthm2} applies to $\fF_s$ with $m=6$.

\end{document}